\documentclass[12 pt]{article}
\usepackage{graphicx,mathtools,amsmath,amssymb,amsthm,enumerate, hyperref, xcolor,mathrsfs}
\usepackage{marvosym, fontawesome}
\usepackage[indent = 20 pt]{parskip}

\usepackage{url}
\usepackage{enumitem}

\newtheorem{theorem}{Theorem}[section]
\newtheorem{corollary}[theorem]{Corollary}
\newtheorem{lemma}[theorem]{Lemma}
\newtheorem{proposition}[theorem]{Proposition}

\newtheorem{letterthm}{Theorem}

\newtheorem{letterques}[letterthm]{Question}

\theoremstyle{definition}
\newtheorem{definition}[theorem]{Definition}
\newtheorem{notation}[theorem]{Notation}

\theoremstyle{remark}

\newcommand{\odd}{odd}
\newcommand{\even}{even}

\newcommand{\cB}{\mathcal{B}}

\newcommand{\cG}{\mathcal{G}}

\newcommand{\cY}{\mathcal{Y}}
\newcommand{\cR}{\mathcal{R}}
\newcommand{\cS}{\mathcal{S}}

\newcommand{\cV}{\mathcal{V}}
\newcommand{\s}{\text{s}}
\newcommand{\nsi}{\text{ns}}

\newcommand{\N}{\mathbb{N}}

\newcommand{\G}{\mathfrak{G}}
\newcommand{\Perm}{\text{Perm}}

\newcommand{\actson}{\curvearrowright}

\DeclareMathOperator{\id}{id}

\DeclareMathOperator{\Iso}{Iso}

\usepackage[doi=false, url =false , isbn = false,style=alphabetic,sorting=nyt, backend = biber, maxcitenames=50, maxalphanames = 5, maxnames=50]{biblatex}
\makeatletter
\def\thanks#1{\protected@xdef\@thanks{\@thanks
        \protect\footnotetext{#1}}}
\makeatother

\begin{document}

\title{Measured groupoids beyond equivalence relations and group actions}
\author{
Soham Chakraborty \thanks{
\hspace{-2 em} \faMapMarker\hspace{0.18 em}: KU Leuven, Department of Mathematics, Leuven, Belgium \\ 
\Letter : \texttt{soham.chakraborty@kuleuven.be, soham.chakraborty.math@gmail.com}
}
} 

\date{\today}

\setlength{\parindent}{0em}

\maketitle 

\begin{abstract}\noindent
We construct the first examples of \textit{genuine} ergodic discrete measured groupoids that are not isomorphic to any equivalence relation or transformation groupoid. We use a construction due to B.H. Neumann of an uncountable family of pairwise non-isomorphic 2-generated groups for our result.

\end{abstract}

\section{Introduction}

Discrete measured groupoids were introduced by Mackey in \cite{Mackey62} and provide a uniform framework to study discrete countable groups, countable measured equivalence relations and (measurable) actions of countable discrete groups on measure spaces. In \cite{Hahn78}, Hahn studied the left regular representation of such groupoids and introduced their von Neumann algebras. Measured groupoids are of interest from an operator algebraic point of view because they give a huge class of examples of von Neumann algebras, that covers group von Neumann algebras, von Neumann algebras with Cartan subalgebras and group measure space von Neumann algebras. Every discrete measured groupoid comes equipped with an associated countable measured equivalence relation and a measured field of isotropy groups. 

As in the case of groups and group actions, several properties of the associated von Neumann algebra is completely captured by the groupoid. In \cite{BCDK24} we introduced the \textit{infinite conjugacy class (icc)} property for groupoids and proved that the von Neumann algebra associated to a discrete measured groupoid is a factor if and only if the groupoid is icc and satisfies an ergodicity condition. In some sense, the conditions as well as the proof of the main theorem in \cite{BCDK24} resemble the cases of groups and group actions. Moreover, the literature on discrete measured groupoids lacks concrete examples beyond equivalence relations and transformation groupoids. This naturally leads
to the question of how much more general the class of ergodic groupoids is compared to the subclass of ergodic equivalence relations and transformation groupoids arising from ergodic actions. More precisely, are there ergodic discrete measured groupoids that are not isomorphic to equivalence relations or transformation groupoids? 

This leads to the natural definition of genuine discrete measured groupoids. We call an ergodic groupoid \textit{genuine} if it is not isomorphic to a transformation groupoid or an equivalence relation. More generally, a discrete measured groupoid is called \textit{genuine} if its direct integral decomposition into ergodic groupoids (as in \cite{Hahn78}) has a positive measure part consisting of ergodic genuine groupoids. This article deals with the following question. 

\begin{letterques}
    Are there genuine discrete measured groupoids? More generally, are there genuine discrete measured groupoids that admit any 
 countable measured equivalence relation? 
\end{letterques}

In this article, we answer this question in the positive for ergodic equivalence relations on a diffuse standard probability space. Moreover, we show that ergodic groupoids with atomic unit spaces are always transformation groupoids, hence answering the question in the negative for such equivalence relations. To summarize, in Theorem \ref{Thm: main theorem ergodic} we prove the following result: 

\begin{letterthm}
    Let $\cR$ be an ergodic equivalence relation on a standard probability space $(X,\mu)$. Then the following are equivalent: 
        \begin{enumerate}
            \item The probability measure $\mu$ is diffuse.
            \item There exists a genuine discrete measured groupoid $\cG$ with unit space $(X,\mu)$ and associated equivalence relation $\cR$. 
        \end{enumerate}
\end{letterthm}

The main idea comes from the following easy observation: If $G \ltimes X$ is the transformation groupoid of a non-singular group action $G \actson (X,\mu)$, the isotropy group at a point $x \in X$ is the stabilizer $\{g \in G  \; | \; gx = x\}$ of the action, and in particular a subgroup of $G$. Thus if we can construct a measured field of groups $(G_{x})_{x \in X}$ such that they cannot be subgroups of a single discrete countable group, then such a measured field of groups cannot be the isotropy part of a transformation groupoid. To construct such a measured field, we use B.H. Neumann's construction from \cite{Neumann37} of an uncountable family of 2-generated groups that are pairwise non-isomorphic. 

The interesting point about Neumann's construction is that he associates to every strictly increasing sequence of odd integers $U$, a countable set of elements $X_{U}$ and permutations $\alpha_{U},\beta_{U} \in \Perm(X_U)$ in a way such that the groups $G_{U}$ generated by these permutations are mutually non-isomorphic. We can use the Borel bijection between $(0,1)$ and $\mathbb{N}^{\mathbb{N}}$ to get a standard Borel structure on the set $Y$ of all strictly increasing sequences of odd integers. Due to the canonical nature of the construction, we can endow $(G_{U})_{U \in Y}$ with a standard Borel structure that makes it a Borel field of groups as in Corollary \ref{Corr: G_U forms a Borel field of countable discrete groups}. Now, by equipping $Y$ with any diffuse Borel probability measure, we get a measured field of groups which can not be the isotropy part of any transformation groupoid. 

For an ergodic equivalence relation $\cR$ on a diffuse standard probability space $(X,\mu)$, we use a basis of Borel bisections (as in \cite[Definition 3.2]{BCDK24}) to construct a new measured field of groups $(H_{x})_{x \in X}$ such that $H_{x}$ is the countable direct sum of $\{G_{y} \}$ where $y$ ranges over the entire orbit of $x$. This gives a new field of groups where $H_{x}$ is isomorphic to $H_{y}$ whenever $(x,y) \in \cR$. This allows us to construct an action of $\cR$ on the measured field of groups $(H_{x})_{x \in X}$ such that the semidirect product is not a transformation groupoid. Using a wreath product construction, we can even construct genuine ergodic groupoids satisfying the icc property, so that the associated von Neumann algebras are factors. In Corollary \ref{Corr: groupoids of type I are not genuine}, we finally show that our result extends to all discrete measured groupoids (not necessarily ergodic) and we can construct genuine groupoids whenever the associated equivalence relation is not of type I.  

We remark that in the proof of Theorem \ref{Thm: main theorem ergodic}, we only use a measured field of groups that is pairwise non-isomorphic and finitely generated (and Neumann's construction provides an example of such a field). One interesting observation is that our construction fails if we restrict ourselves to a class of discrete countable groups that always embed into one universal group. One example is the class of \textit{locally finite groups}, i.e., the ones where every finitely generated subgroup is finite. There is a universal locally finite group called \textit{Hall's universal group} that contains a copy of every locally finite group. It would be interesting to investigate the following question: 

\begin{letterques}
    Are there genuine discrete measured groupoids such that the isotropy groups are locally finite almost everywhere? If so, are there genuine discrete measured groupoids with any prescribed equivalence relation, such that the isotropy groups are locally finite almost everywhere? 
\end{letterques}

\textbf{Acknowledgements:} The author was supported by FWO research project G090420N of the Research Foundation Flanders from October 2020 until September 2024, during which time this article was written. The author would like to thank his PhD supervisor Stefaan Vaes for some insightful comments on the article. The author would also like to thank Milan Donvil, Se-Jin Kim and Felipe Flores for some useful discussions and comments.

\section{Preliminaries}

\subsection{Discrete measured groupoids}

A \textit{discrete Borel groupoid} $\cG$ is a groupoid with a standard Borel structure such that the unit space denoted by $\cG^{(0)}$ is a Borel subset and the source and target maps, denoted by $s$ and $t$ respectively are countable-to-one Borel maps. Letting $X = \cG^{(0)}$, if $\mu$ is a Borel probability measure on $X$ then the source and target maps induce two $\sigma$-finite measures $\mu_{s}$ and $\mu_{t}$ on $\cG$ respectively given by: 
\begin{align*}
    \mu_{s} (A) &= \int_{X} \# \{A \cap s^{-1}(x) \} \; d\mu(x) \\ \mu_{t}(A) &= \int_{X}\# \{A \cap t^{-1}(x) \} \; d\mu(x)
\end{align*}

Such a discrete Borel groupoid $\cG$ will be called a \textit{discrete measured groupoid with unit space $(X,\mu)$} if the measures $\mu_{s}$ and $\mu_{t}$ are equivalent, i.e., they have the same null sets. Two such groupoids $\cG_1$ and  $\cG_{2}$ with unit spaces $(X_1,\mu_1)$ and $(X_2,\mu_2)$ are called \textit{isomorphic} if there is a nonsingular Borel isomorphism $\phi: \cG_{1} \rightarrow \cG_{2}$ that preserves multiplication and inverses. A subset $B$ of a discrete measured groupoid $\cG$ such that $s|_{B}$ and $t|_{B}$ are both injective is usually called a \textit{bisection}. If $s|_{B}: B \rightarrow \cG^{(0)}$ and $t|_{B}: B \rightarrow \cG^{(0)}$ are moreover surjective and hence both bijections, then such a bisection is called a \textit{global bisection}. In this context we usually work with \textit{Borel bisections}, i.e., bisections that are Borel subsets of $\cG$.

The most basic example of a discrete measured groupoid is a countable discrete group with a unit space consisting of a single point. More generally for any such groupoid and any point $x \in X$, the \textit{isotropy group at $x$} is the discrete countable group $\Gamma_x = \{g \in \cG \; | \; s(g) = t(g) = x\}$. The subgroupoid $\Iso(\cG)$ given by $\{g \in \cG \: | \; s(g) = t(g)\}$ is called the \textit{isotropy subgroupoid of $\cG$} or the \textit{isotropy group bundle}. The isotropy subgroupoid consists of a measured family of discrete countable groups $(\Gamma_{x})_{x \in X}$, which can be made precise as follows: 

\begin{definition}
\label{Def: measured field of countable discrete groups}
  \begin{enumerate}
      \item Let $G = (G_{z})_{z \in Z}$ be a family of countable discrete groups indexed over a standard Borel space $Z$. Then $G$ is called a \textit{Borel field of countable discrete groups} if there is a standard Borel structure on the disjoint union $\bigsqcup_{z \in Z}G_{z}$ (which we still denote by $G$) such that the projection map $\pi: G \rightarrow Z$, the multiplication map $G \times G \rightarrow G$, the inverse map $G \rightarrow G$ and the identity section $Z \rightarrow G$ are all Borel maps. 
      
      \item If $\eta$ is a Borel probability measure on $Z$, then $(G_{z})_{z \in Z}$ is called a \textit{measured field of discrete countable groups} over $(Z,\eta)$ if there is a $\eta$-conull subset $Z_{0} \subseteq Z$ such that with the restricted Borel structure on $\pi^{-1}(Z_{0})$,  $(G_{z})_{z \in Z_{0}}$ is a Borel field of discrete countable groups. 
  \end{enumerate}  
\end{definition}

We remark here that while defining a Borel fields of Polish groups $K = (K_{x})_{x \in X}$ in the traditional way (see \cite{Sutherland85}), one also requires the existence of enough Borel sections, i.e., a sequence of Borel sections $\theta_{n} : X \rightarrow K$ such that at each $x \in X$, $\{\theta_{n}({x}) \; | \; n \in \mathbb{N}\}$ is dense in $K_{x}$. Such a sequence of sections is called a \textit{total set of sections}. However note that in the case of discrete countable groups, this condition is automatic, essentially due to the Lusin-Novikov uniformization theorem (see \cite[Theorem 18.10]{Kec95}).

Any such measured field of discrete countable groups forms a discrete measured groupoid. Conversely given any discrete measured groupoid $\cG$ with unit space $(X,\mu)$, the isotropy bundle $(\Gamma_{x})_{x \in X}$ is a measured field of discrete countable groups, which we also call the isotropy subgroupoid $\Iso(\cG)$ depending on the context.

Recall that for a sequence of countable discrete groups $H_{n}$, the direct sum $\bigoplus_{n} H_{n}$ is the group consisting of all sequences of elements $(h_{n})_{n}$ such that $h_{n} = \id_{n} \in H_{n}$ for all but finitely many $n$. Recall that for groups $H$ and $K$, the wreath product $H \wr K$ is the semidirect product group formed from the action $K \actson \bigoplus_{K} H$ given by $k \cdot (h)_l = (h)_{kl}$. The following two lemmas are straightforward and we leave the proofs as exercises.   

\begin{lemma}
    \label{Lemma: sums of fields of groups}
    Let $(Z,\eta)$ be a standard probability space and let $G^{n} = (G^{n}_{z})_{z \in Z}$ be a countable sequence of measured fields of discrete countable groups. Then the countable direct sum $G = \bigoplus_{n} G^{n} = (\bigoplus_{n}(G^{n}_{z}))_{z \in Z}$ is also a measured field of discrete countable groups. 
\end{lemma}

\begin{lemma}
\label{Lemma: wreath products}
    Let $(Z, \eta)$ be a standard probability space and let $(G_{z})_{z \in Z}$ be a measured field of discrete countable groups. Let $K$ be a discrete countable group. Then the wreath products $(G_{z} \wr K)_{z \in Z}$ form a measured field of discrete countable groups. 
\end{lemma}

Another important class of examples of discrete measured groupoids come from countable measured equivalence relations. An equivalence relation $\cR$ on a standard Borel space $X$ is called \textit{countable Borel} if it is a Borel subset of $X \times X$ and every $\cR$-orbit is countable. If $\mu$ is a Borel probability measure on $X$, then $\cR$ is called a \textit{non-singular countable measured equivalence relation on $(X,\mu)$} if for every non-null Borel subset $E \subseteq X$, the $\cR$-saturation $\cR(E) = \{y \in X \; | \; (y,x) \in \cR \text{ for some } x \in E\}$ is non-null. 

In this case the projections $\pi_{s}$ and $\pi_{t}$ sending an element $(x,y)$ to $y$ and $x$ respectively are countable-to-one maps and send Borel sets to Borel sets. It turns out that the non-singularity condition above is equivalent to $\mu_{s}$ and $\mu_{t}$ being equivalent measures on $\cR$ (see \cite[Theorem 2]{Feldman-Moore-1}). Any discrete measured groupoid $\cG$ on $(X,\mu)$ gives rise to a countable measured equivalence relation $\cR_{\cG}$ given by: 
\begin{align*}
    \cR_{\cG} = \{(t(g),s(g)) \; | \; g \in \cG\}
\end{align*}

Orbit equivalence relations are the most basic examples of equivalence relations. Let $G$ be a countable discrete group and $G \actson (X,\mu)$ be a nonsingular action of $G$ on a standard probability space. Then the equivalence relation $\cR_{G \actson X}$ given by $\{(gx,x) \; | \; g \in G, x \in X\}$ is a countable measured equivalence relation called the \textit{orbit equivalence relation}. As shown in \cite[Theorem 1]{Feldman-Moore-1}, every countable measured equivalence relation is in fact isomorphic to an orbit equivalence relation of a countable group action. 

With any such action, one can also associate its transformation groupoid which intersects in some situations with the orbit equivalence relations as we see now. Let $G \actson (X,\mu)$ be a non-singular action on a standard probability space and consider the groupoid $X \rtimes G = \{(x,g) \; | \; x \in X, g \in G\}$ with unit space $(X,\mu)$ such that $s(x,g) = x$ and $t(x,g) = gx$. Multiplication and inverses are defined naturally as follows for all $g,h \in G$ and $x \in X$:  
\begin{align*}
    (gx,h)(x,g) = (x,hg) \text{ and } (g,x)^{-1} = (gx,g^{-1})
\end{align*}
The isotropy group $\Gamma_x$ at $x \in X$ is then isomorphic to the stabilizer $\{g \in G \; | \; gx = x\}$. If the stabilizer groups are trivial almost everywhere, i.e. when the action is \textit{essentially free}, then $X \rtimes G$ is isomorphic to the orbit equivalence relation as a discrete measured groupoid, where the isomorphism is given by $(x,g) \mapsto (gx,x)$. However the converse is not true, i.e., there are countable measured equivalence relations which do not arise as orbit equivalence relations of essentially free actions, and are hence not isomorphic to transformation groupoids, as shown by Furman in \cite{Furman99}. 

Recall from \cite{BCDK24} that a discrete measured groupoid $\cG$ is called \textit{icc} if any non-null Borel subset $E \subset \cG$ satisfies the property that whenever it's conjugacy class $\Omega_{E} = \bigcup_{g} gEg^{-1}$ has finite measure then $E \subset \cG^{(0)}$. By the main theorem in \cite{BCDK24}, an ergodic discrete measured groupoid $\cG$ is icc if and only if the groupoid von Neumann algebra $L(\cG)$ is a factor.

\subsection{Equivalence relations and ergodic decomposition}
\label{Sec: equivalence relations and ergodic decomposition}

For detailed proofs of most of the facts stated in this section, we refer the reader to \cite{Feldman-Moore-1}. Let $\cR$ be a countable measured equivalence relation on a standard probability space $(X,\mu)$. Recall that such a countable measured equivalence relation is called \textit{ergodic} if every $\cR$-invariant Borel subset $E \subset X$ is either $\mu$-null or $\mu$-conull. 

Recall that a function $f \in L^{\infty}(X,\mu)$ is called $\cR$-\textit{invariant} if $f$ is constant on $\cR$-orbits $\mu$-almost everywhere. The space $L^{\infty}(X,\mu)^{\cR}$ of $\cR$-invariant functions is in fact an abelian von Neumann algebra, which is usually denoted by $L^{\infty}(X,\mu)^{\cR}$. It can be checked that $\cR$ is ergodic if and only $L^{\infty}(X,\mu)^{\cR} = \mathbb{C} \cdot 1$. 

Countable measured equivalence relations are broadly classified into three \textit{types}, analoguous to the classification of von Neumann algebras. Let $\cR$ be an ergodic equivalence relation on a standard probability space $(X,\mu)$. Then $\cR$ is said to be of \textit{type I$_{n}$} for $n \in \{1,2,...,\infty\}$ if $\cR$ is isomorphic to the complete equivalence relation on a set of $n$ elements. If $\cR$ is not necessarily ergodic, then $\cR$ is said to be of type I$_{n}$ if $\cR$ is isomorphic to an equivalence relation on $E \times S$, where $E$ is a standard Borel space and $S$ is a set of $n$ elements given by $(x,k) \sim (y,l)$ if and only $x = y$. An equivalence relation is said to be \textit{smooth} or \textit{discrete} or simply \textit{of type I} if $X$ admits a $\cR$-invariant Borel partition $X = \bigsqcup_{n} X_{n}$ such that $\cR|_{X_{n}}$ is of type I$_{n}$ for $n \in \{1,2,...,\infty\}$.

If $\cR$ is not of type $I$, then $\cR$ is said to be of \textit{type II$_{1}$} if there is a probability measure $\nu$ on $X$ which is equivalent to $\mu$ and $\cR$ is $\nu$-preserving, i.e., $\nu_{s} = \nu_{t}$ on $\cR$. If no such equivalent probability measure exists but there is a $\sigma$-finite infinite measure $\nu$ equivalent to $\mu$ such that $\cR$ is $\nu$-invariant, then $\cR$ is said to be of \textit{type II$_{\infty}$}. If there is no invariant finite or infinite measure on $X$ that makes $\cR$ invariant, then $\cR$ is said to be of \textit{type III}. If $\cR$ is a countable measured equivalence relation on $(X,\mu)$, then $X$ decomposes uniquely (up to null sets) into $\cR$-invariant Borel subsets $X_{\text{I}}$, $X_{\text{II}}$ and $X_{\text{III}}$ such that $\cR|_{X_{k}}$ is of type $k$ for $k \in \{\text{I, II, III}\}$. For the purposes of this article, we shall only be working with the decomposition of $\cR$ into its smooth (type I) and non-smooth (non-type I) part. We fix the following notation: 

\begin{notation}
\label{not: smooth and non smooth}
    Let $\cR$ be a countable measured equivalence relation on a standard probability space $(X,\mu)$. Then we shall denote by $X_{\s}$ and $X_{\nsi}$ the $\cR$-invariant Borel subsets of $X$ such that $\cR_{\s} \coloneqq \cR|_{X_{\s}}$ and $\cR_{\nsi} \coloneqq \cR|_{X_{\nsi}}$ are the smooth and non-smooth parts respectively. We will call such an equivalence relation \textit{smooth} if $X_{\s} = X$ and \textit{non-smooth} if $X_{\nsi} = X$ up to measure zero. For a discrete measured groupoid $\cG$ with unit space $(X,\mu)$ and associated equivalence relation $\cR$, we also denote by $\cG_{\s}$ and $\cG_{\nsi}$ the Borel partition of $\cG$ into $\cG|_{X_{s}}$ and $\cG|_{X_{\nsi}}$.  
\end{notation}

Type I equivalence relations are precisely the ones which admit a so-called \textit{fundamental domain}. More precisely we have the following characterization. One can look at standard texts or lecture notes on Borel equivalence relations for a proof, for example see \cite[Chapter 2]{CalderoniLectureNotes09} 

\begin{theorem}
    Let $\cR$ be a countable measured equivalence relation on a standard probability space $(X,\mu)$. Then the following are equivalent: 
    \begin{enumerate}
        \item $\cR$ is of type I,
        \item $\cR$ admits a Borel selector, i.e. a Borel map $f: X \rightarrow X$ satisfying $(x,f(x)) \in \cR$ and $(x,y) \in \cR \implies f(x)=f(y)$ for a.e. $x,y \in X$, 
        \item The space of orbits $X/\cR$ is a standard Borel space.
    \end{enumerate}
\end{theorem}

In Point 2 of the previous theorem, the Borel set $f(X)$ is called a \textit{fundamental domain} for the equivalence relation $\cR$.

A discrete measured groupoid $\cG$ is called \textit{ergodic} if the associated countable measured equivalence relation is ergodic. Suppose now that $(Z,\eta)$ is a standard probability space and $(\cG_{z})_{z \in Z}$ is a measurable field of discrete measured groupoids on a standard measurable space $Z_{0}$ with the canonical projection map $\pi: \bigsqcup_{z \in Z} G_{z} \rightarrow Z$, quasi invariant with respect to the measurable field of Borel probability measures $(\mu_{z})_{z \in Z}$. Let $\mu$ be the integral of the Borel probability measures $(\mu_{z})_{z \in Z}$ with respect to $\eta$. Then there is a natural standard Borel structure on the set $\{(g,\pi(g)) \; | \; g \in \bigsqcup_{z \in Z} G_{z} \rightarrow Z \}$ which makes it a discrete measured groupoid with unit space $(Z_{0} \times Z, \mu)$. The source and target maps are defined as $s(g,\pi(g)) = (s(g),\pi(g))$ and $t(g,\pi(g)) = (t(g),\pi(g))$. Similarly the inverse and composition maps are defined as $(g, \pi(g))^{-1} = (g^{-1}, \pi(g))$ and $(g,\pi(g))(h, \pi(h)) = (gh, \pi(gh))$ whenever $\pi(g) = \pi(h)$ and $s(g) = t(h)$. This discrete measured groupoid is called the \textit{direct integral} of the measurable field $(\cG_{z})_{z \in Z}$ and we shall denote this by $\cG = \int^{\oplus}_{z \in Z} \cG_{z}$.

By the ergodic decomposition theorem (\cite[Theorem 6.1]{Hahn78}), every discrete measured groupoid is isomorphic to a direct integral of ergodic discrete measured groupoids. More precisely, given a discrete measured groupoid $\cG$, let $\cG = \cG_{\s} \bigsqcup \cG_{\nsi}$ be its decomposition into its smooth and non-smooth parts. Then $\cG_{\nsi}$ is isomorphic to a direct integral of ergodic non-smooth groupoids over some standard Borel space $(Y_{\text{ns}})$ and $\cG_{\s}$ admits a decomposition $\cG_{\s} = \bigsqcup_{n \in \mathbb{N}} \cG_{\s}^{n}$, where $\cG_{\s}^{n}$ has associated equivalence relation of type I$_n$ and $\cG_{\s}^{n}$ is a direct integral of ergodic discrete measured groupoids of type I$_{n}$ over standard Borel spaces $Y_n$. The ergodic decomposition of discrete measured groupoids is furthermore unique up to Borel isomorphisms of the standard Borel spaces $Y_n$ and $Y_{\text{ns}}$. For more details we refer the reader to \cite[Theorem 6.1]{Hahn78}. This motivates the following main definition of this article. 

\begin{definition}
    \label{Def: genuine groupoids}
    Let $\cG$ be an ergodic discrete measured groupoid. Then $\cG$ is called \textit{genuine} if $\cG$ is not isomorphic to an equivalence relation or a transformation groupoid. 

    More generally let $\cG$ be a discrete measured groupoid and let $(Z,\eta)$ be a standard probability space such that $(\cG_{z})_{z \in Z}$ is the ergodic decomposition of $\cG$. Then $\cG$ is called \textit{genuine} if there is a non-null Borel subset $E \subset Z$ such that $\cG_{z}$ is genuine for all $z \in E$. 
\end{definition}

\subsection{Semi-direct product groupoids}

Suppose that we have a countable measured equivalence relation $\cR$ on a standard probability space $(X,\mu)$ and suppose $G = (G_{x})_{x \in X}$ is a measured field of discrete countable groups. Then an action of $\cR$ on $G$ is given by a family of group isomorphisms $\delta_{(y,x)}: G_{x} \rightarrow G_{y}$ for all $(y,x) \in \cR$ such that outside a null set we have: 
\begin{enumerate}
    \item $
    \delta_{(z,y)} \circ \delta_{(y,x)} = \delta_{(z,x)}$.
    \item $
    \delta_{(y,x)}^{-1} = \delta_{(x,y)}$
    \item $\delta_{(x,x)} = \id_{x}$
\end{enumerate}

Now the semi-direct product groupoid $\cG = G \rtimes \cR$ consists of elements of the form $(g, (y,x))$ for all $g \in G_{x}$ and $(y,x) \in \cR$. One should think about these elements as the `arrows' from $x$ to $y$, hence we have $s(g,(y,x)) = x$ and $t(g,(y,x)) = y$. The Borel structure in $\cG$ is the product of the Borel structures in $G$ and $\cR$. The multiplication and inverse maps are defined as follows: 
    \begin{align*}
        (h,(z,y)) \circ (g,(y,x)) &= (\delta_{(x,y)}(h)\cdot g, (z,x)) \\ (g,(y,x))^{-1} &= (\delta_{(y,x)}(g^{-1}),(x,y))
    \end{align*}
Since $\cR$ is countable and the projection map $\cG \rightarrow \cR$ is Borel, we have that the source and target maps are countable-to-one and Borel. Similarly, $\cG$ satisfies the non-singularity assumption because the equivalence relation $\cR$ is non-singular. 

It was shown in \cite[Proposition 6.5]{Popa-Shlyakhtenko-Vaes20} that for every so-called treeable equivalence relation $\cR$, any discrete measured groupoid with associated equivalence relation $\cR$ can be written as such a semi-direct product. In particular, this is the case if $\cR$ is a hyperfinite equivalence relation, i.e., one that can be written as a countable union of equivalence relations with finite orbits. For our purposes, we only need the fact that if $\cR$ is a type I equivalence relation (and hence hyperfinite), then any discrete measured groupoid $\cG$ with associated equivalence relation $\cR$ can be written as such a semi-direct product. 

\subsection{B. H. Neumann's construction}
\label{Sec: Neumann's construction}


Recall that a discrete countable group $G$ is called \textit{$n$-generated} for a positive integer $n$ if $G$ can be generated by $n$ elements as a group. B.H. Neumann constructed an uncountable family of discrete countable groups that are mutually non-isomorphic and are all 2-generated in \cite{Neumann37}. For a reference, we direct the reader to \cite[Section III.B]{delaHarpe2000}. We recall the construction here. For all integers $n \geq 1$, let $A_{n}$ denote the alternating group on $n$ elements. Of course the number of elements in $A_{n}$ is $\frac{n!}{2}$ for all $n \geq 1$. It is a standard fact from group theory (for example see \cite[Chapter IV.6]{DummitFoote04}) that for $n \geq 5$, the alternating group $A_{n}$ is simple. 

Recall that for a finite set of elements $a_{i}$ for $1 \leq i \leq n$, the \textit{cycle} $(a_{1},a_{2},...,a_{n})$ denotes the permutation that sends $a_{i}$ to $a_{i+1}$ for all $i < n$ and sends $a_{n}$ to $a_{1}$. Now let $U = \{u_{1},u_{2},...\}$ be a strictly increasing infinite sequence of odd integers with $u_1 \geq 5$. Let $X_{U}$ be a countable set whose elements are labeled as $x_{j,k}$ where $j \geq 1$  and $1 < k \leq u_{j}$. Define two permutation $\alpha_{U}$ and $\beta_{U}$ of the elements of $X_{U}$ given by: 
\begin{align*}
    \alpha_{U} &= (x_{1,1},x_{1,2},...,x_{1,u_{1}})(x_{2,1},x_{2,2},...,x_{2,u_{2}})(x_{3,1},x_{3,2},...,x_{3,u_{3}})... \\
    \beta_{U} &= (x_{1,1},x_{1,2},x_{1,3})(x_{2,1},x_{2,2},x_{2,3})(x_{3,1},x_{3,2},x_{3,3})...
\end{align*}

Notice that $\alpha_{U}$ is a permutation of infinite order but $\beta_{U}$ has order 3. Let $\G_{U}$ be the group of permutations of $X_{U}$ generated by $\alpha$ and $\beta$. By analyzing finite normal subgroups of $\G_{U}$, it turns out that for all $i \geq 1$, $\G_{U}$ has a finite normal subgroup isomorphic to $A_{u_{i}}$. Moreover, every finite normal subgroup of $\G_{U}$ is isomorphic to $A_{u_{i}}$ for some $i \geq 1$. For a detailed proof of these facts, we refer the reader to \cite[Points 26 and 27, Chapter III.B]{delaHarpe2000}. As an immediate corollary, we have the following proposition. 

\begin{proposition} \textup{(\cite[Chapter III.B, Proposition 28]{delaHarpe2000})}
    \label{Prop: uncountable family of non isomorphic groups}
    For two distinct sequences $U$ and $V$ consisting of strictly increasing odd integers $\geq 5$, the 2-generated groups $\G_{U}$ and $\G_{V}$ as constructed above are non-isomorphic. 
\end{proposition}

This gives an uncountable family of 2-generated groups that are mutually non-isomorphic. From an ergodic theoretic view point, this construction can be done entirely in a `Borel manner' and we shall see in the next section that it in fact forms a Borel field of countable discrete groups.

\section{A Borel field of pairwise non-isomorphic groups}

Let $\cY = \N^\N$ be the Baire space with the product topology. It is well known that $\cY$ is a Polish space and hence has a standard Borel structure. It is also well known that there is a continuous order-preserving bijection from $\cY$ to $(0,\infty)$. Hence, with respect to the Borel $\sigma$-algebras on both sides, this bijection is Borel.  We shall denote sequences in $\N^\N$ by $U = (u_{n})_{n \in \N}$. For positive integers $k,l$, consider the subset $\cY_{k,l} \subset \cY$ given by $\{U \in \cY \; | \; u_{k} = l\}$, i.e., all sequences whose $k^{\text{th}}$ entry is $l$. Also consider the subset $\cY_{0} \subset \cY$ consisting of all strictly increasing sequences of odd integers $\geq 5$.

\begin{lemma}
    \label{Lemma: Borel subsets of Y}
    The subsets $\cY_0$ and $\cY_{k,l}$ for all positive integers $k,l \in \N$ are Borel subsets of $\cY$. 
\end{lemma}
\begin{proof}
    For a positive integer $i$, let $\pi_i : \cY \rightarrow \N$ be the map $U \mapsto u_{i}$. By definition the product topology on $\cY$ is the weakest topology for which $\pi_i$ is continuous for all $i$. Since $\cY_{k,l} = \pi_{k}^{-1}(\{l\})$, it is an open set and hence clearly Borel. Now let $\N_{\even}$ be the set of even positive integers. Clearly $E = \bigcup_{i \in \N}\pi_i^{-1}({\N_{\even}})$ is a countable union of open sets and is open and Borel. Then $E^{C} = \cY \backslash E$ is also Borel (note that $E^C$ consists of all sequences with only odd integers). For all pairs $i,j \in \mathbb{N}$ with $i \neq j$, we have that $(\pi_i,\pi_j): E^C \rightarrow \N \times \N$ is continuous and hence Borel. Let $F_{i}$ be the subset of $E^{C}$ consisting of sequences $(u_{n})_{n \in \mathbb{N}}$ such that $u_{i} \geq u_{i+1}$. Then $F_{i}$ is Borel because it is the inverse image under $(\pi_i, \pi_{i+1})$ of the Borel set $\{(m,n) \; | \;  m \geq n\}$ in $\mathbb{N} \times \mathbb{N}$. Then we notice that $E^{C} \backslash \bigcup_{i \geq 1} F_{i}$ consisting of all strictly increasing sequences of odd integers is a Borel set. Now we can subtract the Borel sets $\pi_{1}^{-1}(\{3\})$ and $\pi_{1}^{-1}(\{1\})$ to get that $\cY_{0}$ is indeed a Borel set. 
\end{proof}

\begin{notation}
    For the rest of this article, we shall denote by $Y$ the Borel set $\cY_{0}$ from Lemma \ref{Lemma: Borel subsets of Y}, i.e., the standard Borel structure on the set of all strictly increasing sequences with odd integers $\geq 5$ and we denote by $Y_{k,l}$ be the Borel subset $\cY_{k,l} \cap \cY_{0}$ from Lemma \ref{Lemma: Borel subsets of Y}, i.e., all sequences in $\cY_0$ such that the $k^{\text{th}}$ entry is $l$ for all $k \in \N$ and $l \in \N_{\odd}$.
\end{notation}

For all $U \in Y$, let $X_{U}$ be the countable set defined by B.H. Neumann as in Section \ref{Sec: Neumann's construction} and let $X = \bigsqcup_{U \in Y} X_{U}$. Let $\pi: X \rightarrow Y$ be the canonical countable-to-one projection map. For all positive integers $i,j$, let $\theta_{i,j}: Y \rightarrow X$ be the section given by: 
\begin{align*}
    U \mapsto \begin{cases}
        x_{i,j}   & \text{ if } j \leq u_i \\ x_{i,u_{i}} & \text{ if } j > u_{i}
     \end{cases}
\end{align*}

One can check that the sections $\theta_{i,j}$ have the property that for all $U \in Y$, the set $\{\theta_{i,j}(U) \; | \; i,j \in \N\}$ equals $X_{U}$. Also for two tuples $(i,j) \neq (k,l)$ of positive integers, one can check that the maps $Y \rightarrow [0,\infty)$ given by $U \mapsto d_{u}(\theta_{i,j}(U), \theta_{k,l}(U))$ where $d_{u}$ is the metric that gives the discrete topology, are all Borel, essentially because the maps $\pi_{i}$ are Borel. Then there is a unique standard Borel structure on $X$ that makes the map $\pi: X \rightarrow Y$ and the countable family of sections $\theta_{i,j}$ Borel, for example using \cite[Proposition 4.2]{WoutersVaes24}.   

\begin{lemma}
\label{Def: alpha and beta are borel sections}
    Let $X = \bigcup_{U \in Y} X_{U}$ be the standard Borel space as above. Consider the permutations $\alpha_{U}$ and $\beta_{U}$ of $X_{U}$ as above. Then the disjoint unions $\alpha: X \rightarrow X$ and $\beta: X \rightarrow X$ are Borel.  
\end{lemma}
\begin{proof}
    For $U \in Y$, let $X^{i}_{U} = \{x_{i,1},x_{i,2},...,x_{i,u_{i}}\}$ and let $X^{i} = \bigsqcup_{U \in Y} X^{i}_{U}$. We have that $X^{i}$ is a Borel subset of $X$ for all $i$ because $X^{i} = \bigcup_{j} \theta_{i,j}(Y)$. Now denote by $X^{i,j}$ the Borel subset $\pi^{-1}(Y_{i,j})$ of $X^{i}$. Notice that it is enough to prove that $\alpha|_{X^{i,j}}: X^{i,j} \rightarrow X^{i,j}$ is Borel for all $i,j \in \mathbb{N}$. We know that $\{\theta_{i,k} \; | \; 1 \leq k \leq j\}$ is a total set of Borel sections on $Y_{i,j}$. To check that $\alpha|_{X^{i,j}}$ is Borel, it is enough to prove that $\alpha \circ \theta_{i,k}$ is a Borel section for all $1 \leq k \leq j$ by \cite[Lemma A.2.6]{Lisethesis}. But this holds because restricted to $X^{i,j}$, one has that $\alpha \circ \theta_{i,k} = \theta_{i,k+1}$ if $k \neq j$ and $\alpha \circ \theta_{i,j} = \theta_{i,1}$. Since $\{X^{i,j} \; | \; i,j \in \mathbb{N}\}$ forms a countable Borel partition of $X$, we have that $\alpha: X \rightarrow X$ is a countable disjoint union of Borel maps and is hence Borel.   

    Similarly one can consider the subsets $W^{i}_{U} = \{x_{i,1},x_{i,2},x_{i,3}\}$ and let $W^{i}$  be the disjoint union $\bigsqcup_{Y \in U} W^{i}_{U}$. Exactly as in the previous paragraph, $W^{i}$ is a Borel subset of $X$ for all $i \in \N$. One has that restricted to $W^{i}$, the set $\{\theta_{i,j} \; | \; 1 \leq j \leq 3\}$ is a total set of Borel sections, and clearly $\beta \circ \theta_{i,j}$ is again Borel. Once again by \cite[Lemma A.2.6]{Lisethesis}, we have that $\beta: X \rightarrow X$ is Borel. 
\end{proof}

For two Borel fields of separable structures $H = (H_{z})_{z \in Z}$ and $K = (K_{z})_{z \in Z}$ and the canonical projections $\pi_{H}: H \rightarrow Z$ and $\pi_{K} : K \rightarrow Z$, we denote by $H \times_{\pi} K$ the Borel set $\{(h,k) \in H \times K \; | \; \pi_{H}(h) = \pi_{K}(k)\}$. The proof of the following corollary relies heavily on the results obtained in \cite{WoutersVaes24}. 

\begin{corollary}
    \label{Corr: G_U forms a Borel field of countable discrete groups}
    The family of countable discrete groups $\G = (\G_{U})_{U \in Y}$ is a Borel field of discrete countable groups. 
\end{corollary}
\begin{proof}
    Consider the map $C: \G \times_{\pi} X \rightarrow X$ given by $(g,x) \mapsto g \cdot x$. For each $g \in \G$, let $C_{g}: X_{\pi(g)} \rightarrow X_{\pi(g)}$ denote the cross section of $C$. By Lemma \ref{Def: alpha and beta are borel sections}, for all $g \in \G$, $C_{g}$ is Borel with respect to the restricted Borel structures on $X_{\pi(g)}$. Let $\{\theta_{i,j}\}$ be the total sequence of Borel sections $Y \rightarrow X$ as before. Then consider the map $\phi: \G \rightarrow Y \times X^{\N \times \N}$ given by: 
    \begin{align*}
        g \mapsto (\pi(g), (C_{g}(\theta_{i,j}(\pi(g))))_{(i,j) \in \N \times \N})
    \end{align*}
    We claim that $\phi(\G)$ is a Borel subset of $Y \times X^{\N \times \N}$. To see this, let $\cS$ be the set of finite words generated by the elements $\alpha$, $\beta$ and their inverses. By Lemma \ref{Def: alpha and beta are borel sections}, and the fact that composition and inverses of Borel bijections are Borel, we have that if $g \in \cS$, then $g: X \rightarrow X$ is Borel. Now consider the Borel map $B_{g} : Y \rightarrow X^{\N \times \N}$ given by $U \mapsto g \cdot (\theta_{i,j}(U))_{i,j}$. For each $g \in \cS$, since $B_{g}$ is a Borel map, the graph of $B_{g}$ is Borel (for example see \cite[Lemma A.9]{TakesakiVolume1}). Now since $\cS$ is countable, we have that the union of the graphs of $\{B_{g} \; | \; g \in \cS\}$ is a Borel subset of $Y \times X^{\N \times \N}$, which is precisely equal to $\phi(\G)$. Now by \cite[Lemma 4.1]{WoutersVaes24}, there is a unique standard Borel structure on $\G$ making the maps $\pi: \G \rightarrow Y$ and the map $C: \G \times_{\pi} X \rightarrow X$ Borel. By \cite[Proposition 4.2]{WoutersVaes24}, we have that every element of $\cS$ is a Borel section $Y \rightarrow \G$. It follows that the multiplication, inverse and identity maps are Borel and that $\G = (\G_{U})_{U \in Y}$ is a Borel field of countable discrete groups.    
\end{proof}

\section{Construction of genuine ergodic groupoids}

Letting $\nu$ be a diffuse Borel probability measure on $Y$, we get a discrete measured groupoid from Corollary \ref{Corr: G_U forms a Borel field of countable discrete groups} which satisfies the property that the isotropy groups are non-isomorphic everywhere. We will now use such a construction to construct genuine groupoids with `any' prescribed ergodic equivalence relation.

Recall the notion of a basis of a discrete measured groupoid from \cite[Definition 3.2]{BCDK24}. Roughly speaking, the idea is that a discrete measured groupoid admits a countable disjoint collection of non-null Borel bisections such that the union is the entire groupoid up to measure zero. It turns out that for ergodic groupoids, one can construct a basis consisting of global bisections. This can be seen for example by \cite[Exercise 18.15]{Kec95} and we give a proof of this in our context as the next proposition. Notice that for non-ergodic groupoids this cannot be true. For example, if we have a disjoint union of two measured fields of groups consisting of finite and infinite groups almost everywhere respectively, then such a discrete measured groupoid can never admit a basis of global bisections.  

\begin{proposition}
    \label{Prop: every ergodic groupoid admits a basis of global bisections}
    Let $\cG$ be an ergodic discrete measured groupoid with unit space $(X,\mu)$. Then there is a basis of global bisections, i.e.,  countable disjoint (up to measure zero) family $B_{n}$ of Borel global bisections such that $\bigsqcup_{n} B_{n} = \cG$ up to measure zero.   
\end{proposition}

\begin{proof}
    As in \cite[Proposition 3.1]{BCDK24}, let $\cB = \{X = B_0, B_1, B_2,B_3,...\}$ be a basis of Borel bisections for the groupoid $\cG$. We will use ergodicity to `patch up' these bisections and construct global bisections in the process. For convenience, let us label the Borel sets $\cG_E = \{g \in \cG \; | \; s(g) \in E\}$, $\cG^E = \{g \in \cG \; | \; t(g) \in E\}$ and $\cG^F_E = \{g \in \cG \; | \; s(g) \in E, t(g) \in F\}$ for Borel subsets $E,F \subseteq X$. We construct a new set of Borel bisections $\cV = \{V_0,V_1,V_2,...\}$ of $\cG$. First let $V_0 = B_0 = X$ and suppose that $V_1,V2,...,V_i-1$ are constructed, we shall now construct $V_i$. Let $\widetilde{V} = \bigsqcup_{k=1}^{i-1}$ and let $V_i^i = B_i \backslash \widetilde{V}$. For $j > i$, recursively define the following Borel sets: 
    \begin{align*}
        E^{j}_{i} &= X \backslash s(\bigsqcup_{k = i}^{j-1} V^{k}_{i}) \\
        F^{j}_{i} &= X \backslash t(\bigsqcup_{k=i}^{j-1} V^{k}_{i}) \\
        V^{j}_{i} &= \cG^{F^{j}_i}_{E^{j}_i} \cap (B_j \backslash \widetilde{V})
    \end{align*}
    Notice that by construction and because the maps $s$ and $t$ are countable-to-one, all of the above sets are Borel. Notice also that $V_{i}^{j}$ and $V_{i}^{k}$ are disjoint up to null sets if $j \neq k$. Now we define $V_{i} \coloneqq \bigsqcup_{j \geq i} V_{i}^{j}$ and notice that $V_{i}$ is still a bisection as $V_i^{j}$ and $V^{k}_{i}$ have disjoint sources and ranges for all $j\neq k$. Notice that by construction we also have that $V_{i}$'s are mutually disjoint up to null sets, and since $\bigsqcup_{n \geq 0} B_{n} \subseteq \bigsqcup_{n \geq 0} V_{n}$, we even have that $\cV$ is a basis for $\cG$. 

    The only thing left to show is that $s(V_{n}) = t(V_n) = X$ up to measure zero, i.e., the bisections are global. Notice that $V_0$ is global and assume that $V_0, V_1, ..., V_{i-1}$ are all global bisections. We will prove that $V_i$ is a global bisection. Since $\cB$ is a basis, if $V_i$ is not a global bisection, it follows that for some non-null Borel subsets $E,F \subseteq X$, we have that $V_j \subseteq \cG^F_E$ for all $j \geq i$ or equivalently $s(\cG \backslash \widetilde{V}) = E$ and $t(\cG \backslash \widetilde{V}) = F$. One can check that $f: X \rightarrow \N$ given by $x \mapsto \#\{s^{-1}(x)\}$ is a Borel map, essentially because $s$ is countable-to-one. Hence $f$ is a $\cG$-invariant Borel map that takes different values on $E$ and $X \backslash E$, which contradicts the ergodicity of $\cG$.
\end{proof}

Now we can use Proposition \ref{Prop: every ergodic groupoid admits a basis of global bisections} to prove that every ergodic equivalence relation on a diffuse standard probability space admits genuine measured groupoids. 

\begin{theorem}
    \label{Thm: main theorem ergodic}
        Let $\cR$ be an ergodic equivalence relation on a standard probability space $(X,\mu)$. Then the following are equivalent: 
        \begin{enumerate}
            \item The probability measure $\mu$ is diffuse (equivalently, $\cR$ is not of type I).
            \item There exists a genuine discrete measured groupoid $\cG$ with unit space $(X,\mu)$ and associated equivalence relation $\cR$. 
            \item There exists a genuine icc discrete measured groupoid $\cG$ with unit space $(X,\mu)$ and associated equivalence relation $\cR$.
        \end{enumerate}
\end{theorem}
\begin{proof}
    We first prove that 2 implies 1 by showing that if $\mu$ is atomic, then any discrete measured groupoid with associated equivalence relation $\cR$ is a transformation groupoid. Since $\cR$ is ergodic and $(X,\mu)$ is atomic, we can assume without loss of generality that $\cR$ is the full equivalence relation on an at most countable set $X$ and $\mu$ is the counting measure. Consequently, the isotropy groups $\Gamma_{x}$ are all isomorphic and let us denote this group by $\Gamma$. 
    
    Since $\cG$ is amenable, by \cite[Proposition 6.5]{Popa-Shlyakhtenko-Vaes20}, $\cG$ can be written as a semi-direct product $\cR \ltimes_{\delta} (\Gamma_{x})_{x \in X}$ where $\delta$ is an action of $\cR$ on the measured field of groups $(\Gamma_{x})_{x \in X}$. Let $G$ be a group with the same cardinality as $X$ and suppose for convenience the unit space of $\cG$ is $G$ with the counting measure. Consider the action of $G \times \Gamma$ on $G$ given by $(h,\gamma) \cdot g = hg$. Consider the transformation groupoid $(G \times \Gamma) \ltimes G$ and the map $\phi: (G \times \Gamma) \ltimes G \rightarrow \cG$ given by: 
    \begin{align*}
        \phi(g,(h,\gamma)) = ((hg,g), \delta_{(g,e)}(\gamma))
    \end{align*}
    We now check that that $\phi$ is indeed a groupoid homomorphism. First note that:
    \begin{align*}
      (hg,(k,\lambda))(g,(h,\gamma)) &= (g,(kh,\lambda\gamma)) \text{ and } \\ \phi(g,(kh,\lambda\gamma)) &= ((khg,g),\delta_{(g,e)}(\lambda\gamma))  
    \end{align*}
    One can check now that: 
    \begin{align*}
      \phi(hg,(k,\lambda)) \phi(g,(h,\gamma)) &= ((khg,hg),\delta_{(hg,e)}(\lambda))((hg,g),\delta_{(g,e)}(\gamma)) \\ &= ((khg,g),\delta_{(g,hg)}\circ \delta_{(hg,e)}(\lambda)\delta_{(g,e)}(\gamma)) \\ &= ((khg,g), \delta_{(g,e)}(\lambda\gamma)) 
    \end{align*}
    Similarly, one has that $(g,(h,\gamma))^{-1} = (hg,(h^{-1},\gamma^{-1}))$ and therefore we have $\phi((g,(h,\gamma))^{-1}) = ((g,hg),\delta_{hg,e}(\gamma^{-1}))$. One can check that
    \begin{align*}
        \phi(g,(h,\gamma))^{-1} = ((hg,g),\delta_{(g,e)}(\gamma))^{-1} = ((g,hg),\delta_{hg,e}(\gamma^{-1})) 
    \end{align*}
    It is easy to note now that $\phi$ is a bijective and since the multiplication and the action $\delta$ is Borel, $\phi$ is a Borel groupoid isomorphism. 

    Now we prove that 1 implies 2, essentially using Proposition \ref{Prop: every ergodic groupoid admits a basis of global bisections} and Corollary \ref{Corr: G_U forms a Borel field of countable discrete groups}. Suppose that $\cR$ is an ergodic equivalence relation on the diffuse standard probability space $(X,\mu)$ and as in Proposition \ref{Prop: every ergodic groupoid admits a basis of global bisections} consider a countable family of bijective Borel functions $\theta_{n}: X \rightarrow X$ such that the global bisections $B_{n} = \{(\theta_{n}(x),x) \; | \; x \in X\}$ covers all of $\cR$. Let $(G^{0}_{x})_{x \in X}$ be a measured field of finitely generated pairwise non-isomorphic groups (for example as in Corollary \ref{Corr: G_U forms a Borel field of countable discrete groups}). Let $K$ be an infinite group and note that by \cite[Proposition 1.3.10]{AnaPopa10} the wreath product $G_{x} = G^{0}_{x} \wr K$ is an icc group for all $x \in X$. By Lemma \ref{Lemma: wreath products}, the family of groups $(G_{x})_{x \in X}$ is a measured field of discrete countable groups. 

    Now consider the family of groups $(H_{x})_{x \in X}$ given by $H_{x} = \bigoplus_{n} G_{\theta_{n}(x)}$ where we let $\theta_{0} = \id$. Since $\theta_{n}$ is Borel for all $n$, we have that $(G_{\theta_{n}(x)})_{x \in X}$ is a measured field of groups for all $n$. By Lemma \ref{Lemma: sums of fields of groups}, we have that $H = (H_{x})_{x \in X}$ forms a measured field of groups as well. For each $x \in X$, let $\phi_{x} : \mathbb{N} \rightarrow X$ be defined as $\phi_{x} (n) = \theta_{n}(x)$. By definition, for each $x\in X$, one gets a bijection $\phi_{x}$ between $\mathbb{N}$ and the orbit $\cR(x)$ of $x$. Now for each $(x,y) \in \cR$, we define a permutation $\sigma_{(y,x)}$ of $\mathbb{N}$ by $\sigma_{(y,x)} \coloneqq \phi_{y}^{-1} \circ \phi_{x}$. By construction, $\theta_{\sigma_{(y,x)}(n)}(y) = \theta_n(x)$ for all $(y,x) \in \cR$ and $n \in \mathbb{N}$. Now consider the following family of group isomorphisms $(\delta_{(y,x)})_{(y,x) \in \cR}$, where $\delta_{(y,x)}: H_{y} \rightarrow H_{x}$ sends an element $(h_{n})_{n \in \mathbb{N}} \in H_{y}$ to $(h_{\sigma_{(y,x)}(n)})_{n \in \mathbb{N}}$. 
    
    One can check, almost by construction that this gives an action $\delta: \cR \actson (H_{x})_{x \in X}$ and let $\cG$ be the semidirect product groupoid. We claim that $\cG$ is genuine. Since $\cG$ has non-trivial isotropy everywhere, it is not an equivalence relation. Assume now that $\cG$ is the transformation groupoid of an action $K \actson X$, then $H_{x}$ is a subgroup of $K$ for a.e. $x \in X$ which is a contradiction as $G^{0}_{x} < G_{x} < H_{x}$ are all finitely generated and mutually non-isomorphic. Since $K$ is a countable discrete group, it has at most countably many finite subsets and hence countably many finitely generated subgroups, thus contradicting the assumption and proving the claim. This proves that 1 implies 3 and since 3 clearly implies 2, this completes the proof.   
\end{proof}

If $\cG$ is an icc ergodic discrete measured groupoid with associated equivalence relation $\cR$, then the factor $L(\cG)$ has the same type as the equivalence relation $\cR$. Moreover when $\cR$ is of type III, the flow of weights of $L(\cG)$ corresponds to the associated Krieger flow of $\cR$. Hence Corollary \ref{Thm: main theorem ergodic} gives a class of examples of factors that arise from genuine groupoids with any prescribed type and any prescribed flow of weights. Recall from Section \ref{Sec: equivalence relations and ergodic decomposition} that every discrete measured groupoid decomposes into a direct integral of ergodic ones. This leads to the following easy consequence of Theorem \ref{Thm: main theorem ergodic}.

\begin{corollary}
\label{Corr: groupoids of type I are not genuine}
    Let $\cR$ be a countable measured equivalence relation on a standard probability space $(X,\mu)$. Then the following are equivalent: 
    \begin{enumerate}
        \item $\cR$ is not of type I,
        \item There exists a genuine discrete measured groupoid $\cG$ with associated equivalence relation $\cR$. 
    \end{enumerate}
\end{corollary}

\begin{proof}
    To show that 2 implies 1, one simply notes that if $\cG$ is a discrete measured groupoid of type I, then almost every ergodic groupoid in the ergodic decomposition of $\cG$ is of type I, and hence the statement follows from Theorem \ref{Thm: main theorem ergodic}. To show that 1 implies 2, without loss of generality we can assume that $\cR = \cR_{\nsi}$ as in Notation \ref{not: smooth and non smooth}. Suppose that $(Z,\eta)$ is a standard probability space and $(\cR_{z})_{z \in Z}$ is the ergodic decomposition of $\cR$, where $\cR_{z}$ is an ergodic equivalence relation on a standard probability space $(Z_{0}, \nu_{z})$ for a.e. $z \in Z$. If $(Z, \eta)$ is atomic, then $\cR$ is a disjoint union of countably many ergodic equivalence relations. It follows once again from Theorem \ref{Thm: main theorem ergodic} that there exists a genuine groupoid with associated equivalence relation $\cR$. 
    
    Now suppose that $(Z,\eta)$ is diffuse and let $\nu$ be the integral of the $(\nu_{z})_{z \in Z}$ with respect to $\eta$. We know that $\cR$ can be written as an equivalence relation on $(Z_{0} \times Z, \nu)$ given by $(w,z) \sim (w',z')$ if and only if $z = z'$ and $(w,w') \in \cR_{z}$. Consider a measured field $G = (G_{w})_{w \in Z_{0}}$ of discrete countable finitely generated groups such that they are pairwise non-isomorphic almost everywhere, as in Corollary \ref{Corr: G_U forms a Borel field of countable discrete groups}. Consider the measured field of groups $H = (H_{(w,z)})_{(w,z) \in Z_{0} \times Z}$ defined as $H_{(w,z)} \coloneqq G_{w}$ for all $z \in Z$ and $w \in Z_{0}$. Since $\cR$ is not necessarily ergodic, we do not necessarily have a basis of global bisections as in Proposition \ref{Prop: every ergodic groupoid admits a basis of global bisections}. However we can work around this issue as follows. 
    
    Let $\cV = \{V_{0}, V_{1}, V_{2}, V_{3},...\}$ be a basis for $\cR$ where $V_{0}$ is the unit space as in \cite[Proposition 3.1]{BCDK24}. Suppose that each partial Borel bisection $V_{n}$ is given by a Borel map $\theta_{n}: A_{n} \rightarrow B_{n}$ with $A_{n}$ and $B_{n}$ Borel subsets of $Z_{0} \times Z$. Let $H^{n} = (H^{n}_{(w,z)})_{(w,z) \in Z_{0} \times Z}$ be the Borel field of groups given by $H^{n}_{(w,z)} \coloneqq G_{\theta_{n}(w,z)}$ if $(w,z) \in A_{n}$ and $H^{n}_{(w,z)} \coloneqq \{e\}$ if $(w,z) \notin A_{n}$. Since $A_{n}$ is a Borel subset of $Z_{0} \times Z$, $H^{n}$ is a Borel field of groups for all $n$ with $H^0 = H$. By Lemma \ref{Lemma: sums of fields of groups}, we have that $(K_{(w,z)})_{(w,z) \in Z_{0} \times Z} = (\bigoplus_{n} H^n_{(w,z)})_{(w,z) \in Z_{0} \times Z}$ is again a Borel field of groups. Since $\cR = \cR_{\nsi}$, we have that almost every orbit of $\cR$ is strictly infinite and we can hence construct an action of $\cR \actson (K_{(w,z)})_{(w,z) \in Z_{0} \times Z}$ exactly as in Theorem \ref{Thm: main theorem ergodic} by considering permutations of $\N$. 
    
    Letting $\cG$ be the semidirect product groupoid, one can see almost by construction that $\cG$ admits a direct integral decomposition $\cG = \int^{\oplus}_{Z} \cG_{z}$ where for almost every $z \in Z$, we have that $\cG_{z}$ is the semidirect product groupoid of the restricted action $\cR_{z} \actson (K_{(w,z)})_{w \in Z_{0}}$. Now exactly as in the proof of Theorem \ref{Thm: main theorem ergodic}, we have that $\cG_{z}$ is genuine for almost every $z \in Z$ because $K_{(w,z)}$ contains $G_{w}$ as a subgroup for almost every $w \in W$. Thus we can conclude, by definition that $\cG$ is a genuine discrete measured groupoid with associated equivalence relation $\cR$.    
\end{proof}

\emergencystretch3em
\printbibliography

@article {WoutersVaes24,
    AUTHOR = {Vaes, Stefaan and Wouters, Lise},
     TITLE = {Borel fields and measured fields of {P}olish spaces, {B}anach
              spaces, von {N}eumann algebras, and {$\rm C^*$}-algebras},
   JOURNAL = {J. Lond. Math. Soc. (2)},
  FJOURNAL = {Journal of the London Mathematical Society. Second Series},
    VOLUME = {111},
      YEAR = {2025},
    NUMBER = {4},
     PAGES = {Paper No. e70159},
      ISSN = {0024-6107,1469-7750},
   MRCLASS = {46L10 (03E15 22D05 46B20 46L05)},
  MRNUMBER = {4895946},
       DOI = {10.1112/jlms.70159},
       URL = {https://doi.org/10.1112/jlms.70159},
}

@phdthesis{Lisethesis,
author = {Wouters, Lise},
language = {eng},
title = {Equivariant Jiang-Su stability for actions of amenable groups},
school ={KU Leuven},
year = {2023},
}

@book {Kec95,
    AUTHOR = {Kechris, Alexander S.},
     TITLE = {Classical descriptive set theory},
    SERIES = {Graduate Texts in Mathematics},
    VOLUME = { \; 156},
 PUBLISHER = {Springer-Verlag, New York},
      YEAR = {1995},
     PAGES = {xviii+402},
      ISBN = {0-387-94374-9},
   MRCLASS = {03E15 (03-01 03-02 04A15 28A05 54H05 90D44)},
  MRNUMBER = {1321597},
MRREVIEWER = {Jakub\ Jasi\'{n}ski},
       DOI = {10.1007/978-1-4612-4190-4},
       URL = {https://doi.org/10.1007/978-1-4612-4190-4},
}

@book {TakesakiVolume1,
    AUTHOR = {Takesaki, M.},
     TITLE = {Theory of operator algebras. {I}},
    SERIES = {Encyclopaedia of Mathematical Sciences},
    VOLUME = {124},
      NOTE = {Reprint of the first (1979) edition,
              Operator Algebras and Non-commutative Geometry, 5},
 PUBLISHER = {Springer-Verlag, Berlin},
      YEAR = {2002},
     PAGES = {xx+415},
      ISBN = {3-540-42248-X},
   MRCLASS = {46Lxx (46-01)},
  MRNUMBER = {1873025},
}

@article {Popa-Shlyakhtenko-Vaes20,
    AUTHOR = {Popa, S. and Shlyakhtenko, D. and Vaes, S.},
     TITLE = {Classification of regular subalgebras of the hyperfinite {$\rm
              II_1$} factor},
   JOURNAL = {J. Math. Pures Appl. (9)},
  FJOURNAL = {Journal de Math\'{e}matiques Pures et Appliqu\'{e}es.
              Neuvi\`eme S\'{e}rie},
    VOLUME = {140},
      YEAR = {2020},
     PAGES = {280--308},
      ISSN = {0021-7824,1776-3371},
   MRCLASS = {46L36 (20L05 37A20 46L10 46L55)},
  MRNUMBER = {4124434},
MRREVIEWER = {Qihui\ Li},
       DOI = {10.1016/j.matpur.2020.02.009},
       URL = {https://doi.org/10.1016/j.matpur.2020.02.009},
}

@article {Feldman-Moore-1,
    AUTHOR = {Feldman, Jacob and Moore, Calvin C.},
     TITLE = {Ergodic equivalence relations, cohomology, and von {N}eumann
              algebras. {I}},
   JOURNAL = {Trans. Amer. Math. Soc.},
  FJOURNAL = {Transactions of the American Mathematical Society},
    VOLUME = {234},
      YEAR = {1977},
    NUMBER = {2},
     PAGES = {289--324},
      ISSN = {0002-9947},
   MRCLASS = {22D40 (28A65 46L10)},
  MRNUMBER = {578656},
       DOI = {10.2307/1997924},
       URL = {https://doi.org/10.2307/1997924},
}

@article {Furman99,
    AUTHOR = {Furman, Alex},
     TITLE = {Orbit equivalence rigidity},
   JOURNAL = {Ann. of Math. (2)},
  FJOURNAL = {Annals of Mathematics. Second Series},
    VOLUME = {150},
      YEAR = {1999},
    NUMBER = {3},
     PAGES = {1083--1108},
      ISSN = {0003-486X,1939-8980},
   MRCLASS = {22F10 (22E40 28D15 37A15 53C24)},
  MRNUMBER = {1740985},
MRREVIEWER = {Scot\ Adams},
       DOI = {10.2307/121063},
       URL = {https://doi.org/10.2307/121063},
}

@misc{CalderoniLectureNotes09,
  AUTHOR        = {Calderoni, F.},
  TITLE         = {Lecture notes on countable Borel equivalence relations},
  SCHOOL = {University of Illinois, Chicago},
  YEAR          = {2020},
  URL={https://math.berkeley.edu/~vfr/VonNeumann2009.pdf}
}

@article{Neumann37,
    author = {Neumann, Bernhard Hermann},
    title = {Some remarks on infinite groups},
    journal = {Journal of the London Mathematical Society},
    year = {1937},
    volume = {1},
    number = {2},
    pages = {120-127}
}

@book {delaHarpe2000,
    AUTHOR = {de la Harpe, Pierre},
     TITLE = {Topics in geometric group theory},
    SERIES = {Chica-go Lectures in Mathematics},
 PUBLISHER = {University of Chicago Press, Chicago, IL},
      YEAR = {2000},
     PAGES = {vi+310},
   MRCLASS = {20F65 (20F69 57M07)},
  MRNUMBER = {1786869},
MRREVIEWER = {Lee\ Mosher},
}

@book {DummitFoote04,
    AUTHOR = {Dummit, David S. and Foote, Richard M.},
     TITLE = {Abstract algebra},
   EDITION = {Third},
 PUBLISHER = {John Wiley \& Sons, Inc., Hoboken, NJ},
      YEAR = {2004},
     PAGES = {xii+932},
      ISBN = {0-471-43334-9},
   MRCLASS = {00-01 (16-01 20-01)},
  MRNUMBER = {2286236},
}

@article{BCDK24,
      title={Factoriality of groupoid von Neumann algebras}, 
      author={Tey Berendschot and Soham Chakraborty and Milan Donvil and Se-Jin Kim},
      year={2024},
      eprint={2402.04449},
    journal = {to appear in J. Operator Theory},
    archivePrefix={arXiv},
      primaryClass={math.OA}
}

@unpublished{AnaPopa10,
    author = {Anantharaman-Delaroche, Claire and Popa, Sorin},
    title = {An introduction to II$_{1}$ factors},
    note = {Preprint: https://www.math. ucla.edu/popa/ Books/IIun.pdf},
    year = {2010},
}

@article {Mackey62,
    AUTHOR = {Mackey, George W.},
     TITLE = {Point realizations of transformation groups},
   JOURNAL = {Illinois J. Math.},
  FJOURNAL = {Illinois Journal of Mathematics},
    VOLUME = {6},
      YEAR = {1962},
     PAGES = {327--335},
      ISSN = {0019-2082},
   MRCLASS = {28.65 (22.40)},
  MRNUMBER = {143874},
MRREVIEWER = {H.\ L.\ Royden},
       URL = {http://projecteuclid.org/euclid.ijm/1255632330},
}

@article {Hahn78,
    AUTHOR = {Hahn, Peter},
     TITLE = {The regular representations of measure group-oids},
   JOURNAL = {Trans. Amer. Math. Soc.},
  FJOURNAL = {Transactions of the American Mathematical Society},
    VOLUME = {242},
      YEAR = {1978},
     PAGES = {35--72},
      ISSN = {0002-9947,1088-6850},
   MRCLASS = {46L10 (28D99 46K15)},
  MRNUMBER = {496797},
       DOI = {10.2307/1997727},
       URL = {https://doi.org/10.2307/1997727},
}

@article {Sutherland85,
    AUTHOR = {Sutherland, Colin E.},
     TITLE = {A {B}orel parametrization of {P}olish groups},
   JOURNAL = {Publ. Res. Inst. Math. Sci.},
  FJOURNAL = {Kyoto University. Research Institute for Mathematical
              Sciences. Publications},
    VOLUME = {21},
      YEAR = {1985},
    NUMBER = {6},
     PAGES = {1067--1086},
      ISSN = {0034-5318},
   MRCLASS = {22A05 (46L10)},
  MRNUMBER = {842412},
MRREVIEWER = {G. A. Reid},
       DOI = {10.2977/prims/1195178509},
       URL = {https://doi.org/10.2977/prims/1195178509},
}

\end{document}